\documentclass[11pt]{amsart}

\usepackage[dvips]{graphicx}
\usepackage{calc}
\usepackage{color}
\usepackage{amsmath}
\usepackage{amssymb}
\usepackage{amscd}
\usepackage{amsthm}
\usepackage{amsbsy}
\usepackage{delarray}
\usepackage{enumerate}
\usepackage[T1]{fontenc}
\usepackage{inputenc}
\usepackage{enumerate}
\usepackage{hyperref}

\usepackage{amscd}
\usepackage{amsthm}
\usepackage{array}
\usepackage{indentfirst}
\usepackage{geometry}
\usepackage{epsfig}

\usepackage{dsfont}
\usepackage{tikz}
\usetikzlibrary{matrix,arrows}
\usepackage{mathrsfs}
\usepackage[all]{xy}

\DeclareMathOperator{\cHom}{\mathscr{H}\mspace{-3mu}om}

\DeclareMathOperator{\supp}{supp}

\DeclareMathOperator{\osc}{osc}

\renewcommand{\leq}{\leqslant}
\renewcommand{\geq}{\geqslant}

\newcommand{\ko}{{\bf k}}

\newcommand{\eq}{\begin{eqnarray}}
\newcommand{\eneq}{\end{eqnarray}}
\newcommand{\eqn}{\begin{eqnarray*}}
\newcommand{\eneqn}{\end{eqnarray*}}

\newcommand{\isoto}[1][]{\xrightarrow[#1]%
{{\raisebox{-.6ex}[0ex][-.6ex]{$\mspace{1mu}\sim\mspace{2mu}$}}}}

\def\osc{\mathrm{osc}}

\def\epi{\mathrm{epi}}

\title{Homological differential calculus} 
\author{Nicolas Vichery}
\date{\today}

\newtheorem{thm}{Theorem}[section]
\newtheorem{lemma}[thm]{Lemma}
\newtheorem{prop}[thm]{Proposition}
\newtheorem{coroll}[thm]{Corollary}

\theoremstyle{definition}

\newtheorem{definition}[thm]{Definition}

\newtheorem{question}[thm]{Question}
\newtheorem{notation}[thm]{Notation}

\newtheorem{property}[thm]{Property}

\newtheorem*{acknow}{Acknowledgement}

\newlength{\espaceavantspecialthm}
\newlength{\espaceapresspecialthm}
\newenvironment{exemple}[1][]{\refstepcounter{thm} 
\vskip \espaceavantspecialthm \noindent \textsc{Example~\thethm
#1.} }%
{\vskip \espaceapresspecialthm}

\newenvironment{rem}[1][]{\refstepcounter{thm} 
\vskip \espaceavantspecialthm \noindent \textsc{Remark~\thethm #1.} }%
{\vskip \espaceapresspecialthm}

\newcommand{\R}{{\mathbb{R}}}

\newcommand{\Z}{{\mathbb{Z}}}
\newcommand{\N}{{\mathbb{N}}}

\newcommand{\cD}{{\mathcal{D}}}

\newcommand{\cF}{{\mathcal{F}}}

\newcommand{\cS}{{\mathcal{S}}}

\newcommand{\fc}{{:\ }}

\begin{document}

\maketitle

\begin{abstract}              
This article provides a definition of a subdifferential for continuous functions based on homological considerations.  We show that it satisfies all the requirement for a good notion of subdifferential. Moreover, we prove sublinearity, a Leibniz formula and an approximation result.
This work fits in the framework of microlocal analysis of sheaves for $C^0$ symplectic problems and application to Aubry-Mather theory.
\end{abstract}

\section{Introduction}

The main part of this article is based on the field of homological algebra and more specifically on the microlocal analysis of sheaves. We provide a notion of subdifferential on the space of lower semi-continuous functions from a closed smooth manifold $X$ to $\R$.
There already exists a wide range of subdifferentials and we include this new one to the list checking some general condition a subdifferential should satisfy (Definition \ref{condition}). 
We prove sublinearity and Leibniz formula for the homological subdifferential. We will finish with an approximation statement of the subdifferential of a  $C^0$ convergent sequence of function.

The Proposition \ref{eqdef} gives to the reader not  familiar with the notion of microsupport an alternative definition of our subdifferential based only on singular homology when we deal with continuous functions.
That is the main reason why this theory will be developed only for continuous functions. We will also compare our sudifferential (named homological) with the approximate subdifferential and Clarke subdifferential in the Lipschitz case.

We believe that $C^0$ symplectic geometry and relatives could benefit from this new object. Indeed, all the concepts developped here fit very well in the framework for Hamiltonian non-displeacability that Tamarkin introduces in \cite{Tamarkin}. This work is highly based on the book of Kashiwara-Schapira \cite{KS} where they introduced the notion of microsupport of sheaves and systematically studied its properties. A study of the link of such a framework with classical objects of symplectic topology will appear in a subsequent paper of the author \cite{whatsheaves}.

The more striking result should come from the theory of Aubry-Mather for non-convex Hamiltonian system \cite{nonconvex}. If the Hamiltonian is Tonelli (convex in $p$), the Mather's $\alpha$ functional is convex \cite{Fathi}. But, in the non-convex case, following the definition of \cite{MVZ} the homogenized Hamiltonian does not have to be convex and even differentiable. Moreover, it can be obtained as a $C^0$ limit of a sequence of continuous functions constructed from symplectic capacities. A classical object of study in such a field is the subdifferential of the $\alpha$ function.  Even if all notions of subdifferential are the same in the convex case (see Definition \ref{condition}). In light of \ref{limit}, the homological subdifferential seems to be the right object to perform a generalization of the theory to non convex Hamiltonian. 

The main idea of homological subdifferential is based on a well known fact in Morse theory. If there is no critical point between two sublevel sets then the inclusion morphism is an isomorphism. This idea is similar to the one used to define the singular support of a sheaf. Indeed, this notion measures 'microlocaly' the change of cohomology and can be interpreted as a microlocal Morse deformation indicator.

\begin{acknow}
I would like to warmly thank Claude Viterbo for sharing his conception of critical point for continuous function that inspired section \ref{eqdef}. I am also grateful to Pierre Schapira for discussions and to Stéphane Guillermou for patiently explaining to me microlocal theory of sheaves and for giving to me precious remarks on a preliminary version. Finally, I would like to thank Vincent Humil\`ere for advice on a previous draft.

The research leading to these results has received funding
from the European Community's Seventh Framework Progamme ([FP7/2007-2013] [FP7/2007-2011]) under grant agreement $\text{n}\textsuperscript{o}$ [258204].
\end{acknow}

\section{Review on singular support of sheaves}

In this section, we recall some definitions and results from
\cite{KS}. We consider a smooth real manifold $M$.

\subsection{Geometrical part  and notations (\cite{KS})}

One denotes by   $\pi\fc T^*M\to M$ the
cotangent bundle of $M$. We equip this bundle with the symplectic form $\omega:=d\lambda$
with $\lambda$ the  canonical Liouville form.

The antipodal map $a_M$ is defined by:
\eq\label{eq:antipodal}
&&a_M\fc T^*M\to T^*M,(x;\xi)\mapsto(x;-\xi).
\eneq
If $A$ is a subset of $T^*M$, we denote by $A^a$ its image by the antipodal map.

 Let $L\subset M$ be a smooth submanifold. We
denote by $\nu^*_LM$ its conormal bundle  defined by the following exact sequence:

\[0\to \nu^*_LM\to L\times_MT^*M\to T^*L\to0 \ .\]

We identify $M$ to $\nu^*_MM$ and set $\dot{T^*M}:= T^*M\setminus M$.

Let $f\fc M\to N$ be  a smooth map. 

To $f$ are associated the cotangent morphisms

\eq\label{diag:cotgmor}
&&\xymatrix{
T^*M\ar[d]^-\pi&M\times_N\ar[d]^-\pi\ar[l]_-{f_d}\ar[r]^-{f_\pi}T^*N
                                      & T^*N\ar[d]^-\pi\\
M\ar@{=}[r]&M\ar[r]^-f&N.
}\eneq
One sets 
\eqn
&&\nu^*_MN:=\ker f_d= {f_d}^{-1}(\nu^*_MM). 
\eneqn

\begin{rem}\label{remgraph}
Let $\Gamma_f$ be the graph of $f$ in $M\times N$, 
the projection $T^*(M\times N)\to M\times T^*N$ identifies 
$\nu^*_{\Gamma_f}(M\times N)$ and $M\times_N T^*N$.
\end{rem}

\begin{definition}

Let $S_1,S_2\subset M$. Their Whitney normal cone,
denoted by $C(S_1,S_2)$, is the closed cone of $TM$ defined  in local coordinates by:

If $(x;v)$ is the associated
coordinate system on $TM$, then 
\eqn
&&\left\{
\parbox{60ex}{
$(x_0;v_0)\in C(S_1,S_2)\subset TM$ if and only
if there exists a sequence $\{(x_n,y_n,c_n)\}_n\subset S_1\times S_2\times\R^+$ such
that $x_n\stackrel{n}{\rightarrow}x_0$, $y_n\stackrel{n}{\rightarrow}x_0$  and $c_n(x_n-y_n)\stackrel{n}{\rightarrow}v_0$.
}
\right.
\eneqn

\end{definition}

\begin{definition}
Let $S\subset M$ and let $L\subset M$ a smooth submanifold, 
the Whitney normal cone of $S$ along $L$, denoted $C_L(S)$,
is the image in $T_LM$ of $C(L,S)$.

\end{definition}

\begin{definition}\label{normcone}
Let $S$ be a subset of $M$. We define the strict normal cone by

\[N_x(S):=T_x M\setminus C_x(M\setminus S, S)\]

\noindent and the conormal cone by the polar cone (denoted by $.^\circ$) to the strict normal cone at $x$:

\[N^*_x(S)=N_x(S)^\circ\]

\noindent and

\[N^*(S)=\bigcup_{x\in M} N^*_x(S)\ .\]
\end{definition}

We shall use the Hamiltonian isomorphism 
$H\fc T^*(T^*M)\isoto T(T^*M)$ inverse of the isomorphism induced by $\omega$.

\begin{definition}{\rm (see~\cite[Def. 6.5.1]{KS})}
\label{def:coisotropic}
A subset $S$ of $T^*M$ is coisotropic, 
at $p\in T^*M$ if for any
$\theta\in \nu^*_pT^*M$ such that the  Whitney normal cone $C_p(S,S)$ is
contained in the hyperplane $\{v\in TT^*M;\langle v,\theta\rangle=0\}$, 
one has $-H(\theta)\in C_p(S)$.  A set $S$ is coisotropic if it is so
at each $p\in S$. 
\end{definition}

When $S$ is a smooth submanifold, one recovers the usual notion of coisotropic submanifold. This means that at each point of $S$, the symplectic orthogonal of the tangent space is included into the tangent space.

\begin{definition}
Let $A\subset T^*M$. The subset $A$ is said to be $\R^+-$conic if $A$ is invariant by positve dilatation in the fiber of $T^*M$.
\end{definition}

\begin{definition}
Let $f\fc M\to N$ be a smooth map and let 
$\Lambda\subset T^*N$ be a closed $\R^+$-conic subset. One says that 
$f$ is non-characteristic for $\Lambda$ ( or else, $\Lambda$ 
is non-characteristic for $f$, or $f$ and $\Lambda$ are transversal )\, if
\eqn
&&{f_\pi}^{-1}(\Lambda)\cap \nu^*_MN\subset M\times_N\nu^*_NN.
\eneqn
\end{definition}

A morphism $f\fc M\to N$ is non-characteristic for a closed
$\R^+$-conic subset $\Lambda$ of $T^*N$ if and
only if $f_d\fc M\times_NT^*N\to T^*M$ is proper on ${f_\pi}^ {-1}(\Lambda)$
and in this case 
$f_d{f_\pi}^{-1}(\Lambda)$ is closed and $\R^+$-conic in $T^*M$.

\subsubsection{Conification}

Since our theory is highly linked with symplectic and contact geometry, it is classical to consider the jet space of $X$.

\begin{definition}
The jet space $J^1(X)$ is the contact manifold given by:
 \[J^1(X):=T^*X \times \R\]
\noindent with the contact one form $\tilde\lambda=\lambda+dt$.
\end{definition}
 
\begin{rem}
We intentionally take this unusual sign convention in order to keep notation as light as possible.
\end{rem} 

The manifold $T^*X\times T^*\R$ equipped with the $1$-form $\tau dt+\lambda$ and whose associated symplectic form is  $\omega_{T^*X}+d\tau\wedge dt$ is the symplectified of $J^1(X)$. We thus define:

\[\xymatrix{\{\tau>0\}\cap T^*(X\times\R) \ar[rd]^{\tilde \rho} \ar[rr]^\rho && T^*X \\ & J^1(X) \ar[ru]^r}\]

\noindent with
\[\rho(x,t,p,\tau)=(x,\frac p \tau )\ ,\]
\[\tilde \rho(x,t,p,\tau)=(x,\frac p \tau ,t)\ ,\]
\noindent and $r$ the canonical projection.
\begin{definition}

Let $L \subset T^*X$ a smooth Lagrangian, we associate a smooth manifold of dimension $n+2$ in $T^*X\times T^*\R$:

\[Cone(L)=\rho^{-1}(L)\ .\]

\end{definition}

\begin{rem}
In semi-classical analysis, this conification corresponds to the introduction of a parameter $\hbar$ which has the nature of an "action".
\end{rem}

\begin{notation}
For all conic subsets $A\subset T^*X\times T^*\R$, we denote \[Red(A)=\rho(A\cap \{\tau>0\})\ .\]
\end{notation}

\begin{rem}
The map $Red$ is the symplectic reduction along the coisotropic submanifold $\{(x,p,t,\tau)\in T^* X\times T^*\R,\tau=1\}$.
\end{rem}

\begin{property}
 $Red(Cone(L))=L\subset T^*X$.
\end{property}

\subsection{Microsupport}

We consider a commutative unital ring $\ko$ of finite global dimension 
({\em e.g.} $\ko=\Z/2\Z$).
We denote by $D(\ko_M)$ the derived category of sheaves of $\ko$-modules on $M$.

Recall the definition of the microsupport (or singular support) $SS(F)$ of a sheaf $F$.

\begin{definition}{\rm (see~\cite[Def.~5.1.2]{KS})}
Let $F\in D^b(\ko_M)$ and let $p\in T^*M$. 
One says that $p\notin SS(F)$ if there exists an open neighborhood
$U$ of $p$ such that for any $x_0\in M$ and any
real $C^1$-function $\phi$ on $M$ defined in a neighborhood of $x_0$ 
satisfying $d\phi(x_0)\in U$ and $\phi(x_0)=0$, one has
$(R\Gamma_{\{x;\phi(x)\geq0\}} (F))_{x_0}\simeq0$.
\end{definition}

The following properties are well known:

\begin{itemize}
\item
The microsupport is closed and invariant by the action of  $\R^+$ on $T^*M$. 
\item
$SS(F)\cap \nu^*_MM =supp(F)$.
\item
Suppose that $F_1\to F_2\to F_3\stackrel{[+1]}{\to}$ is a
distinguished triangle in  $D^b(\ko_M)$, then 
$SS(F_i)\subset SS(F_j)\cup SS(F_k)$ for all $i,j,k\in\{1,2,3\}$
with $j\not=k$. 
\end{itemize}

\begin{thm}{\rm (see~\cite[Th.~6.5.4]{KS})}
Let $F\in D^b(\ko_M)$. Then its microsupport 
$SS(F)$ is coisotropic.
\end{thm} 

\begin{definition}
Let $Z$ a locally closed subset of $M$. The constant sheaf on $Z$, denoted by $\ko_Z$, is the sheaf on $M$ with the following section group :

\[\ko_Z(U):=\left\{ f\fc U\cap Z \rightarrow k,\ \text{locally constant and}\ supp(f)\  \text{closed in}\ U  \right\}\]
\end{definition}

\begin{exemple}\label{fundex}

\begin{enumerate}[a)]
 \item  If $F$ is a non-zero local system on a connected manifold $M$,
then $SS(F)$ is the zero-section.

\item If $Z$ is a smooth closed submanifold of $M$ and $F=\ko_Z$, then 
$SS(F)=\nu^*_Z M$.

\item Let $\phi$ be a $C^1$-function with $d\phi(x)\not=0$ when $\phi(x)=0$.
Let $Z=\{x\in M;\phi(x)\geq0\}$. 
Then 

\[SS(\ko_Z)=Z\times_M\nu^*_MM\cup\{(x;\lambda d\phi(x));\phi(x)=0,\lambda\geq0\}\ .\]

\end{enumerate}

\end{exemple}

We recall a fundamental bound on the microsupport of sheaves constant on a closed subset.

\begin{prop}\cite[Prop.-5.3.8]{KS}\label{bc}
 Let M be a manifold and Z a closed subset. Then:
 
 \[SS(\ko_Z)\subset N^*(Z) \ .\]

\end{prop}

We now consider the case where $M=X\times \R$ with $X$ a smooth manifold.

\begin{definition}
Let $F\in D^b(k_{X\times\R})$. The representative of $F$ is the set \[R(F):=Red(SS(F))\subset T^*X \ .\]
\end{definition}

\subsubsection{Functorial operations (proper and non-characteristic cases)}

Let $M$ and $N$ be two real manifolds. We denote by $q_i$ ($i=1,2$)
the $i$-th projection defined on $M\times N$ and by $p_i$ ($i=1,2$)
the $i$-th projection defined on $T^*(M\times N)\simeq T^*M\times T^*N$.

\begin{thm}\label{th:opboim}{(See \cite[\textsection 5.4]{KS}.)}
Let $f\fc M\to N$ be a morphism of manifolds,
let $F\in D^b(\ko_M)$ and let $G\in D^b(\ko_N)$.  
\begin{enumerate}
\item
One has
\eqn
&& SS(F\boxtimes G)\subset SS(F)\times SS(G),\\
&& SS(R\cHom({q_1}^{-1} F, {q_2}^{-1}G))\subset SS(F)^a\times SS(G).
\eneqn 
\item
Assume that $f$ is proper on $\supp(F)$. Then
$SS(Rf_*F)\subset f_\pi{f_d}^{-1}SS(F)$.
\item
Assume that $f$ is non-characteristic with
respect to $SS(G)$. Then 
the natural morphism ${f}^{-1} G\otimes \omega_{M/N}\to {f}^{!}(G)$
is an isomorphism. Moreover
$SS({f}^{-1}G) \cup SS({f}^{!}G) \subset f_d {f_\pi}^{-1} SS(G)$.

\end{enumerate}
\end{thm}

\begin{coroll}\label{cor:opboim}
Let $F_1,F_2\in D^b(\ko_M)$.
\begin{enumerate}
\item
Assume that $SS(F_1) \cap SS(F_2)^a\subset \nu^*_MM$. Then
\eqn
&& SS(F_1 \stackrel{L}{\otimes} F_2)\subset SS(F_1)+SS(F_2).
\eneqn
\item
Assume that $SS(F_1)\cap SS(F_2)\subset \nu^*_MM$. Then 
\eqn
&& SS(R\cHom(F_1,F_2))\subset SS(F_1)^a+SS(F_2).
\eneqn

\end{enumerate}
\end{coroll}

\begin{rem}

In \cite{KS}, the authors introduce the operation $\hat{+}$ that allowed them to bound also the singular support of a tensor product even if non-characteristicity does not occur. The operation $\hat{+}$ can be decomposed in two terms $+$ and $+_\infty$. 

\end{rem}

\begin{rem}
In the case where $f$ is smooth, we would like to stress that the different previous bounds on the microsupport can be written as Lagrangian correspondences. Considering the bound on the direct image, we look at $B=f_\pi{f_d}^{-1}(A)$ with $A\subset T^*M$. In fact $B$ can also be computed by the Lagrangian correspondence associated to $f$:
\[B=\Lambda_f(A)\]
\noindent where
\[\Lambda_f:=\{(x,\xi,f(x),\nu)\in T^*M\times\overline{T^*N},\ \nu\circ df=\xi\} \ .\]
\end{rem}

\section{Subdifferential definition and requirements}

\subsection{Definitions}

\begin{definition}
Let $f\fc X \to \R$. The epigraph of $f$, denoted by $\epi(f)$, is the subset of $X\times\R$ defined by \[\epi(f):=\{(x,t)\in X\times \R; f(x)\leq t\} \ .\]
\end{definition}

\begin{definition}
Let $f\fc X\to \R$. We denote by $ F_f$ the sheaf $\ko_{\epi(f)}$.
\end{definition}

\begin{rem}
In particular, if $f$ is lower semi-continuous (i.e. $\epi(f)$ closed), then $ F_f$ can be seen as an object of the Tamarkin category, also seen as a subcategory of $D^b(\ko_{X\times\R})$ as defined in \cite{Tamarkin}.

This category is designed such that objects "represents" in some sense (not useful here) coisotropic submanifolds of the cotangent bundle. Here, Lagrangians of interest in $T^*X$ are the graphs of the differential of smooth functions and their non-smooth generalizations.
\end{rem}

\begin{definition}
Let  $f\fc X\to\R$ be a lower semi-continuous function.
The subdifferential of $f$ is \[\partial f:=R({ F_f})^a\ .\]

\end{definition} 

We denote by $\partial f|_x:=\partial f \cap {\nu^*_x X}$, or also $\partial f(x):=\partial f|_x$ if there is no possible confusion.

\begin{exemple}
 Let $f\in C^1(X)$ then $\partial f|_x=\{df(x)\}$ according to example \ref{fundex}.
\end{exemple}

\subsection{An equivalent definition} \label{eqdef}

This section is devoted to give a workable definition of the subdifferential for people not familiar with microlocal study of sheaves. Indeed, the subdifferential can be reduced to the computation of the singular support of a sheaf constant on $\epi(f)$. The particular form of this sheaf allows us to make a formulation of $\partial f$ in terms of local behavior of level set cohomology.

\begin{definition}

Let $f\fc X\to \R$ be continuous.

\begin{itemize}
\item We say that a point $x$ is a singular point of $f$,  if the following morphism is not an isomorphism:

\[\lim\limits_{\stackrel{U\ni x}{\epsilon \to 0}} H^*(U\cap f^{< \epsilon+a}) \to  \lim\limits_{\stackrel{U\ni x}{\epsilon \to 0}} H^*(U\cap f^{< a}) \ where\ a:=f(x) \ .\] 
 
\item We say that $x$ is  a critical point, if there exists a sequence of $(\phi_n,x_n)\in C^1(X)\times X$ such that $f-\phi_n$ is singular at $x_n$, $x_n \to x$ and $d\phi_n(x_n)\to 0$.

\end{itemize}

\end{definition}

\begin{prop}
Let $f\fc X \to \R$. Then $\xi \in \partial f|_x$ if and only if  x is a critical point of $x\mapsto f(x)-\langle \xi,x\rangle$ (defined locally).

\end{prop}

\begin{proof}

Adding $-\langle \xi,.\rangle$ to $f$, we are reduced to $\xi=0$. Then the statement means $(x,f(x);0,1)\in SS(F_f)$.
To show that a point belongs to a microsupport, we have to compute $R\Gamma_{\{\Psi(x,t)\geq 0\}}F_f$ for $\Psi\in C^1(X\times \R)$. The conic property allow us to consider only function such that $\frac{\partial\Psi}{\partial t}=1$.Then, the implicit function theorem reduces locally the set $\{(\psi(x,t)\geq 0\}$ to $\{\phi(x)\leq t\}$.

\noindent At this point, we have to compute the germ of:

\[R\Gamma_{\{\phi(x)\leq t\}} \ko_{\epi(f)} \ .\]

\noindent We use the following distinguished triangle:

\[R\Gamma_{\{\phi(x)\leq t\}} \ko_{\epi{f}} \to R\mathscr{H}om(\ko_{X\times\R},\ko_{\epi{f}}) \to R\mathscr{H}om(\ko_{\{\phi(x)>t\}},\ko_{\epi{f}}) \to  \ .\]

\noindent We apply the diffeomorphism $(x,t)\mapsto (x,t-\phi(x))$ and  obtain the distinguished triangle:

\[R\Gamma(U;R\Gamma_{\{\phi(x)\leq t\}} \ko_{\epi{f}}) \to R\Gamma(\phi(U);R\mathscr{H}om(\ko_{X\times\R},\ko_{\epi{f-\phi}})) \to R\Gamma(\phi(U);R\mathscr{H}om(\ko_{0>t},\ko_{\epi{f-\phi}})) \to  \ .\]

\noindent We deduce from the previous discussion that :

\[R\Gamma_{\{\phi(x)\leq t\}}\ko_{\epi{f}}|_{(x_0,t_0)}\simeq R\Gamma_{\{0\leq t\}} \ko_{\epi{f-\phi}}|_{(x_0,t_0)} \ .\]

Thus, the condition involved in the definition of the microsupport is equivalent to the fact that there exists a sequence $(\phi_n,x_n)\in C^1(X)\times X$ with $\phi_n(x_n)=f(x_n)$, $\phi_n'(x_n)=\xi_n\to 0$, $x_n\to x$ and such that:

\[\lim\limits_{\stackrel{U\ni x_n}{\epsilon \to 0}} H^j(U\times]-\epsilon,\epsilon[,\ko_{\epi(f-\phi_n)}) \to \\  \lim\limits_{\stackrel{U\ni x_n}{\epsilon \to 0}} H^j(U\times ]-\epsilon,0[,\ko_{\epi(f-\phi_n)})\]

\noindent is not an isomorphism.

This is equivalent to the fact that:

\[\lim\limits_{\stackrel{U\ni x_n}{\epsilon \to 0}} H^*(U\cap (f-\phi_n)^{< \epsilon}) \to  \lim\limits_{\stackrel{U\ni x_n}{\epsilon \to 0}} H^*(U\cap (f-\phi_n)^{< 0}) \]

\noindent is not an isomorphism, which means that $x_n$ is a singular point of $f-\phi_n$ at $x$. The point $x$ is then a critical point of $f$.
\end{proof}

\subsection{The operator $\partial$ is a subdifferential}

In the zoology of subdifferentials a sequence of important requirements has emerged in order to call an object subdifferential. Let $\cS(TX)$ be the set of closed subsets of $T^*X$.

\begin{definition} \label{condition}

For an operator $\partial\fc C^0(X)\to \cS(T^*X)$, we define $\partial f|_x:=\partial f \cup \nu_x^*M$. 
The operator $\partial$ is a subdifferential if the following holds:

\begin{enumerate}
\item {substantiability}: \[x \notin dom(f)\Rightarrow \partial f|_x =\emptyset\ ;\]

\item localizability:\[f=g\ \textrm{on a neighborhood of}\ x \Rightarrow\partial f(x)=\partial g(x)\ ;\]
\item {contiguity}: 

	\begin{enumerate}[a)]
	  \item Let f  be a convex function on $\R^n$.  Then $\partial f|_x$ is the projection via $\rho$ of the dual cone to $\epi(f)$ at $x$;
	  \item Let f be strictly Fréchet differentiable. Then \[df(x)\in\partial f |_x\ ;\]
	\end{enumerate}
	
\item {optimality}: If $f$ reaches a local minimum at x then \[0\in \partial f|_x \ ;\] 

\item {calculability} :
  \begin{enumerate}[a)]
  \item Let $f,g\fc \R^n \to \R$ be such that 
  \[g(x)=\lambda f(Ax+b)+<l,x>+\alpha\]
  with $\lambda \alpha \in \R$, $l,b \in \R^n$ and $A\in M_{n\times n}(\R)$. Then
  \[\partial g|_x = \lambda \partial f|_{Ax+b}\circ A+l \ ;\]
  
  \item Let $f\fc X\times Y \to \R$ be such that $f(x,y)=g(x)+h(y)$. Then
  \[\partial f\subset \partial g \times \partial h\ ;\]
  \end{enumerate}

\item {boundness}: Let $g$ be a Riemannian metric on X and $f$ K-Lipschitz. Then \[||x^*||\leq K,\ \forall x^*\in \partial f(x) \ .\] 

\end{enumerate}
\end{definition}

\begin{thm} 
The homological subdifferential $\partial$ is a subdifferential.
\end{thm}

\begin{proof}
We check the properties in the order of the previous definition.
\begin{enumerate}
\item Substantiability is tautological.
\item Localizability is true since the definition of the microsupport is local.
\item Contiguity a). In this case, it has been computed in \cite{KS} that  $SS(F_f)$ is the dual cone of $\epi(f)$.

Contiguity b) is equivalent to boundness and calculability \cite{Ioffe}.

\item Optimality. Suppose that we are at a maximum and $f(x_0)=0$. For a given open ball $U$ containing $(x_0,f(x_0))$ small enough, consider a system of coordinates $(y,t) \in \R^k\times \R$ on $U$ with origin $(x_0,f(x_0))$. Let $\|.\|$ a euclidean norm and $\phi(y,t)=t-\|y\|^2$. Let $V\in B\times ]-\epsilon,\epsilon[$ ($B$ ball centered at $(x_0,f(x_0))$). We get the following morphism: \[H^\bullet(V,{F}_f) \to H^\bullet(V\cup \{\phi<0\},F_f)\ .\]
This last morphism is clearly not an isomorphism, the left hand-side being isomorphic to the cohomology of $B$ and the right handside to the cohomology of $B\setminus\{0\}$. It is not an isomorphism between inductive limits.
According to the definition of the singular support, $(1,0)=\phi^{'}(0,0)\in SS({F}_f)|_{(x_0,f(x_0))}$.

\item Calculability.

a) Except for the multiplication by $\lambda$, all calculability properties can be easily deduced from the behavior of $SS$ under the action of diffeomorphism on the base.

The case of $\lambda=-1$ is the only one that could be difficult to study. But, lemma \ref{opp} gives that $-\partial(f)=\partial(-f)$.

b) First, we have $F_f:=Rs_!(F_g\boxtimes F_h)$ with
\[s\fc X\times Y \times \R \times \R \to X \times Y \times \R\]
\noindent the sum over the $\R$ variables.

By theorem \ref{th:opboim} we have: 

\[ SS(F_f)\subset \Lambda_s(SS(F_g)\times SS(F_h))\ ,\]

where
\[\Lambda_s=\{x,y,t_1,t_2,\xi,\nu,\tau_1,\tau_2,x,y,t_1+t_2,\xi,\nu,\tau=\tau_1=\tau_2\}\ .\]

\noindent We deduce $\partial f=\partial g \times \partial h$.

\item Boundness. Suppose that f is $1$-Lipschitz. Without loss of generality, let $X=\R^n$, $x=0$ and $f(x)=0$. We want to compute a bound for $N^*(\epi(f))$. An easy reformulation and the Lipschitz property shows that :

\[C(\epi(f),T(X\times\R) \setminus \epi(f))\subset \{(v,r), r\leq ||v||\} \ .\]

Thus,

\[\{v,r), r > ||v||\}\subset N(\epi(f))\]

\noindent and
\[ N^*(\epi(f))\subset \{(\xi,\tau), \tau \geq ||\xi||\} \ .\]

According to lemma \ref{cbound}, $SS(F_f)\subset \{(\xi,\tau), \tau \geq ||\xi||\}$. By applying $\rho$, we get the result.

\end{enumerate}
\end{proof}

\begin{rem}\label{nicelip}
 The proof of boundness shows also that $SS(\ko_{\epi(f)})$ can be totally recovered from $\rho(SS(\ko_{\epi(f)}))=\partial f$, the intersection with $\{\tau=0\}$ being included in the zero section of $T^*(X\times\R)$.
\end{rem}

\subsection{Relation with other subdifferentials}

It is possible, in the case where $f$ is locally Lipschitz, to give a relation between Clarke subdifferential and the homological subdifferential.

\begin{definition}
Let $f\fc X \to \R$ a locally Lipschitz map. We denote by $f^\circ(x,v)$ the limit:
\[\limsup\limits_{y\to x,t\downarrow 0} \frac{f(y+tv)-f(y)}{t} \ .\]
The Clarke subdifferential, denoted by $\partial_c f$, is the set :
\[ \{\xi \in \nu_x^*X, \langle \xi,v\rangle \leq f^\circ(x,v)\ \forall v\}\ .\]
\end{definition}

\begin{rem}
 We need to choose local coordinates to define $f^\circ(x,v)$. Nevertheless, $\partial_c f$ does not depend of this choice.
\end{rem}

In the finite dimensional case, Rademacher theorem applies to Lipschitz functions. Clarke, in \cite{Clarke}, gives a very nice formulation of his subdifferential in terms of limits of almost all derivatives.
We denote by $\Omega_f$ the set of points at which $f$ fails to be differentiable.

\begin{thm}[Clarke] \label{cbound}
Let $f$ a function locally Lipschitz around $x$, and suppose $S$ is any set of Lebesgue measure $0\in \R^n$. Then 

\[\partial_c f(x)=co\{ \lim df (x_i): x_i\to x, x_i\notin S, x_i\notin \Omega_f \}\]

\noindent where $co$ designs the convex hull.
 
\end{thm}

\begin{thm}
Let $f$ be a locally Lipschitz function around $x$. Then \[\partial f\subset\partial_c f \ .\]
\end{thm}

\begin{proof}
We will use the bound of the Proposition \ref{bc}.
Let $(v,r)\in C(X\times \R \setminus \epi(f),\epi(f))$ with $(v,r)\neq 0$. This is equivalent to : $\exists x_n, y_n \to x, \exists c_n>0$ such that :

\[\left\{
    \begin{array}{ll}
      (1)\ c_n(x_n-y_n)\to v   \\
      (2)\  \exists a_n, b_n \downarrow 0, c_n(f(x_n)-a_n-f(y_n)-b_n)\to r
    \end{array}
\right.
\]

We denote by $s_n:=a_n+b_n\geq 0$.

Point $(1)$ and $v\neq 0$ implies that $c_n\to \infty$. Thus we can define $t_n:=c_n^{-1}>0\to 0$. We take the notation $v_n:=t_n^{-1}(x_n-y_n) \to 0$.

Point $(2)$ implies $t_n^{-1}(f(y_n+t_nv_n)-f(y_n)-s_n)\to r$.

But, \[t_n^{-1}(f(y_n+t_nv_n)-f(y_n))= 
\underbrace{t_n^{-1}(f(y_n+t_n v_n)-f(y_n+t_n v))}_{o(|v_n-v|)}+t_n^{-1}(f(y_n+t_n v)-f(y_n))\]
 
\noindent which implies that $r\leq f^\circ (x,v)$.
It follows that $N_x(\epi(f))=\{(v,r), r>f^\circ(x,v)\}$.
From \cite{KS}, we compute the polar set of $N_x(\epi(f))$ intersected with $\tau=-1$. Thus, $(-\xi,1) \in N_x(\epi(f))^*$ is equivalent to 

\[ \langle (-\xi,1),(v,r) \rangle \geq 0 \ .\]

Thus, $\xi \in \{\xi, \langle \xi,v \rangle \leq f^\circ (x,v) \}$.

\end{proof}

\begin{rem}
We would like to stress the fact that the definition of the polar set in \cite{KS} differs from the definition in \cite{Clarke} by sign.
\end{rem}

Here is an example such that $\partial f \neq \partial_c f$.

\begin{exemple}

Let $$ \left\{ {\begin{array}{ll} f \fc \R^2 \to \R \\ (x,y)\to |x|+|y|-\sqrt{x^2+y^2} \end{array}} \right. \ .$$ We would like to compute $\partial_c f$ and $\partial f$ at $0$.

First, Clarke's Proposition \ref{cbound} gives us that $\partial_c f$ belongs to the convex hull of \\ $\{\lim \nabla f(x_i,y_i), (x_i,y_i) \to 0  \}$. In this case and  when define, we get : 

\[ d f(x,y) =\left(sign(x)-\frac{x}{\sqrt{x^2+y^2}},sign(y)-\frac{y}{\sqrt{x^2+y^2}}\right) \ .\]

Thus,

\[\{\lim df(x_i,y_i), (x_i,y_i) \to 0  \}= \{(x,y),y=\pm 1 \mp \sqrt{(2-|x|)|x|}\}\]

\begin{center}
	\includegraphics[width=0.30\textwidth]{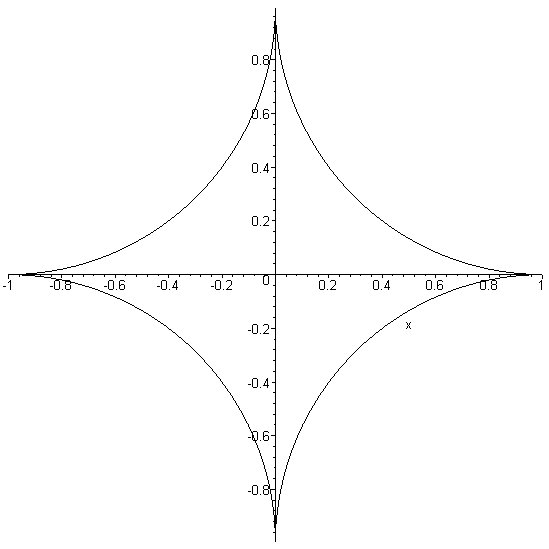}
\end{center}

and,

\[\partial_c f|_{(0,0)}=co(\{(-1,1),(1,1),(1,-1),(-1,-1)\}) \ .\]

It is usually hard to show that a point is not in the singular support, because we have to compute some cohomology group with varying parameter depending of choices of $C^1$ function. Nevertheless, we will use here the result of Appendix \ref{legendre}.

We will compute the germ of the extended Legendre transform of $f$ in a neighborhood of $(\frac{1}{2},\frac{1}{2})$. Indeed, this point belongs to Clarke's subdifferential and we want to show that it does not belong to the homological subdifferential.

The germ at $(a_0,b_0,t_0)$ near $(\frac{1}{2},\frac{1}{2},t_0)$ of $\hat{f}$ is given by the compact cohomology of the set :

\[A:=\{(x,y), |x|+|y|-\sqrt{x^2+y^2}-a_0 x -b_0 y\leq t_0 \} \ .\]

We use the fact that $|x|+|y|-\sqrt{x^2+y^2}-a_0 x -b_0 y$ is homogeneous of degree $1$. If $t_0\geq 0$, $A$ is a closed set null homotopic with $n$ branches going to infinity (see the black part of the next pictures ). 

\begin{center}
\includegraphics[width=0.30\textwidth]{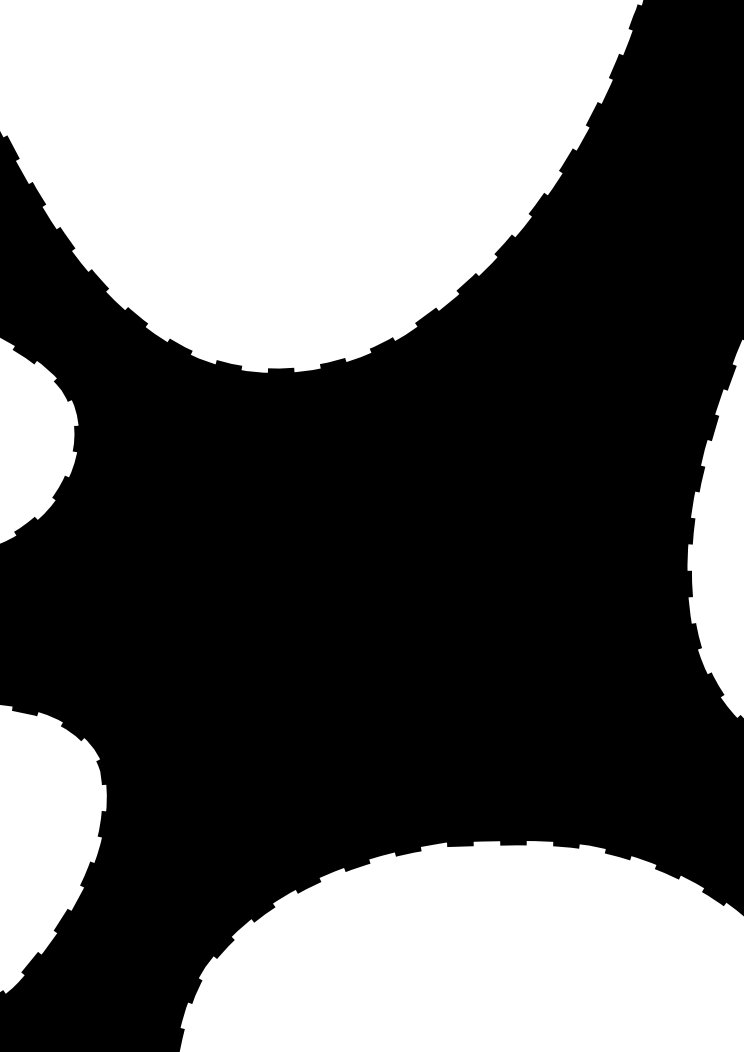}
\end{center}

Thus, $H_c^*(A)=\ko^{n-1}$.

The number $n$ is the cardinal of connected components of $\{t,|cos(t)|+|sin(t)|-1-a_0 cos(t) -b_0 sin(t)>0\}\subset\mathbb{S}^1 $ on a period. A careful study shows that $n=1$, hence, $H_c^*(A)=0$.

For $t_0<0$, $H_c^*(A)=0$.  

Locally, $\hat{f}$ is the null sheaf. Hence, there is no microsupport above $(\frac{1}{2},\frac{1}{2},t)$ for all $t$ and we deduce $(\frac{1}{2},\frac{1}{2}) \notin \partial f$.

This method and the computation of the subdifferential outside $0$ allow us to show that the subdifferential at $0$ is :

\begin{center}
	\includegraphics[width=0.30\textwidth]{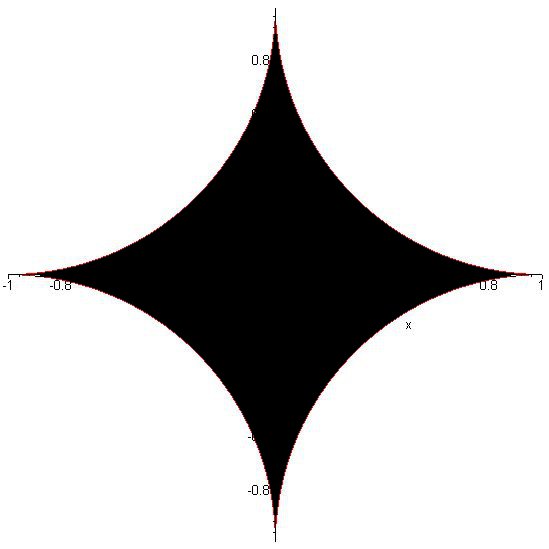}
\end{center}

Another way to prove that Clarke's subdifferential and homological one are different is the following. Consider the sequence of functions : 

\[f_n(x)=|x|^{1+\frac{1}{n}}+|y|^{1+\frac{1}{n}}+(x^2+y^2)^{\frac{1}{2}+\frac{1}{n}} \ .\]

The sequence $f_n$ is $C^1$ and converges to $f$ in $C^0$ topology on every compact.
For a ball of size $1$ around $0$, we look at the union of differentials of $f_n$. For $n$ large, $(\frac{1}{2},\frac{1}{2})$ is not in this set. According to Proposition \ref{limit}, $(\frac{1}{2},\frac{1}{2})$ is not in the subdifferential.

\end{exemple}

\begin{rem}
The previous example shows also that Clarke's differential does not satisfy Proposition \ref{limit} which will be necessary in a future paper on applications to non convex Aubry-Mather theory.
\end{rem}
\begin{rem}
 Stéphane Guillermou pointed out also that the sheaf considered is conic in the sense of \cite{KS}. The traditional Fourier-Sato transform for conic sheaves gives a more straight way of computing its microsupport.
\end{rem}

The homological subdifferential $\partial$ differs from the Clarke one by Proposition\ref{limit}. Nevertheless, according to \cite{jourani}, another non convex subdifferential possesses the upper-limit Proposition \ref{limit}.
It is $\partial_G$ the G-subdifferential also named approximate subdifferential. It is included in the Clarke subdifferential and its convex hull is Clarke's subdifferential.
This subdifferential possesses a minimality property proved by Ioffe in \cite{Ioffe2} :

\begin{thm}[Thm 9 \cite{Ioffe2}]
 Assume that a set $\partial f\subset T^*X$ is associated with all Lipschitz function $f$, such that the following is true:
 \begin{itemize}
  \item $0\in \partial f(x)$ if x is a minimum;
  \item $\partial f|_x=\limsup\limits_{y\to x} \partial f|_y$;
  \item for a convex function $\partial f$ is the usual one;
  \item $\partial (f+g)\subset\partial f+\partial g$ provided that $g$ is convex. 
  \end{itemize}

 Then for every $f$ Lipschitz, $\partial_G f \subset\partial f$.
\end{thm}

Previous discussions, and the closure of the singular support give, according to Ioffe 'stheorem, the next theorem:

\begin{thm}
Let $f\fc X\to\R$ a Lipschitz function. Then,  \[ \partial_G f\subset \partial f \subset co(\partial_G(f))=\partial_C f \ .\]
\end{thm}

\begin{rem}
 The homological subdifferential seems to be the biggest known subdifferential that satisfies proposition \ref{limit}. We can also deduce from The previous theorem the non-emptiness of the homological subdifferential in the Lipschitz case.
\end{rem}

\section{Formulae}

\subsection{Limiting behavior}

The following consideration have roots in the field of $C^0$ symplectic geometry.
In his thesis, Humilière \cite{Humiliere} started the study of the completion of the space of Lagrangian  isotopic through Hamiltonian deformation to the zero section for the spectral metric. This construction is abstract, we can thus ask the question about the existence of a ``geometric support`` its elements which is expected not to be a smooth manifold.

 Moreover as we noted in \cite{KS,whatsheaves}, sheaf theory is well adapted to include non-smooth Lagrangians extending the original definition.
 
We discuss the possibility of writing elements of Humiliere's completion as representatives of objects in $\cD(\ko_{X\times\R})$.

In this paper, we restrict ourselves to a very particular class of Lagrangians that are differential of smooth functions and then to sheaves on $X\times\R$ of the form $\ko_{\epi(f)}$. The spectral distance between $0$ and $graph(df)$ being classically known to be equal to $\osc(f)$.

\begin{exemple}
The sequence of Lagrangians,
 \[L_n:=\left\{ \left( x,\frac {n-1}{n} \frac{\cos(x)}{\sqrt{1-(\frac{n-1}{n})^2 sin^2(x)}} \right),x\in\mathbb{S}^1\right\}\] 
 
\noindent is a Cauchy sequence for the spectral norm. Indeed, it is sufficient to study the sequence of generating functions : \[S_n(x)=arcsin \left(\frac{n-1}{n}sin(x)\right)\] which is a Cauchy sequence for the $C^0$ norm. This sequence has a limit in Humili\`{e}re's completion . However, the $C^0$ limit of the Lagrangians is not smooth. Nevertheless, we note there exists a sheaf $\ko_{arcsin(sin(x))}$ that has for representative a ''crenel''.
\end{exemple}

\begin{question}
Can we find a sheaf  and  a (geometrical) representative associated to elements in Humiliere's completion ?
\end{question}

Denote by $Mod(\ko_X)$ the category of sheaves on $X$.

\begin{lemma}\label{limit_microsupport}

Let $ F_n$ a filtrant inductive system of sheaves in $Mod(\ko_X)$. Then  
\[SS\left(\varinjlim\limits_{n\to\infty}  F_n \right) \subset \lim\limits_{n\to\infty} SS( F_n)\ .\]
\end{lemma}

The limit in the right hand-side being the set of points that are limits of sequences such that $u_n\in SS( F_n)$.

\begin{proof}
We use exercise V.7 of \cite{KS}, which is solved in \cite{Viterbo_course_eil}.
We get the result,
\[SS(\varinjlim  F_n) \subset \overline{\bigcup_n SS( F_n)}.\]
Applying this to truncated system by below, we get: 
\[SS(\varinjlim  F_n) \subset \bigcap_{k}\overline{\bigcup_{n=k}^\infty SS( F_n)} \ .\]
\end{proof}

\begin{prop}\label{limit}
Let $f_n\fc X\to \R$ be a sequence of continuous functions converging to $f$ for the  $C^0$ norm. Then, 
\[\partial f\subset \bigcup\limits_{k>0} \overline{\bigcap\limits_{n>k} \partial f_n}\ .\]

Meaning that for all $\xi\in\partial f|_x$, there exists an extraction of $f_n$, denoted also by $f_n$ and a sequence $x_n$ converging to $x$ such that $\xi$ is in the closure of $\partial f_n|_{x_n}$.
\end{prop}

\begin{proof}

The sequence $||f_n-f_{n+1}||_\infty$ is bounded. Consider a subsequence that by abuse we denote by $f_i$ such that \[a_n:=\sum\limits_{i=n+1}^\infty ||f_i-f_{i+1}||_\infty \stackrel{n\to\infty}{\rightarrow} 0 \ .\] It is always possible to re-normalize $f_n$ adding a sequence of constants $-a_n$ . This implies $f_n\leq f_k$ if $n\leq k$.

We get that ${F}_i:=\ko_{\{f_i\leq t\}}$ is a direct system. Let $ F=\ko_{\{f\leq t\}}$. We consider the natural morphism:
\[\varinjlim {F}_i \to {F}\ .\] But, it is easy to check that this morphism induces an isomorphism on germs of $\varinjlim {F}_i$ and thus on ${F}$.{\color{blue}} This gives: \[{F}\simeq\varinjlim {F}_i\ .\] We use Lemma \ref{limit_microsupport} to conclude.

\end{proof}

This is related to a question asked by the author in his thesis about geometric interpretation of limit of smooth Lagrangian for the spectral distance and solved by Seyfaddini and Viterbo:

\begin{prop}
Let $L_n$ be a converging sequence of Lagrangians for the spectral metric to a smooth Lagrangian $L$. Then : 
$\forall x\in L, \exists (x_n)_{n\in\N}$ such that $x_n\in L_n$ and $\lim\limits_{n\to\infty} x_n=x$.
\end{prop}

\begin{proof}
 Suppose that an open neighborhood $U$ of a point $x\in L$ is not in the closure of the set of Lagrangians. Then, we consider a non trivial Hamiltonian action with support in $U$ denoted by $\phi$. Then by biinvariance $\phi(L_n)=L_n$ converge to $\phi(L)\neq L$ which is false.
\end{proof}

\subsection{Sum of subdifferentials}

We recall from \cite{Tamarkin} the notion of convolution of sheaves.

\begin{definition}
Let \[s\fc X\times\R\times\R\to X\times \R\] the addition on real factors, \[m\fc X\times\R\times\R\to X\times \R\] the multiplication on real factors and $\pi_i\fc X\times\R_1\times \R_2\rightarrow X\times\R_i$ the canonical projection. Then we define : \[ F*G:=Rs_!(\pi_1^{-1} F\stackrel{L}{\otimes}\pi_2^{-1}G)\]
\noindent and
\[F\diamond G:=Rm_!(\pi_1^{-1} F\stackrel{L}{\otimes}\pi_2^{-1}G)\ .\]
\end{definition}

\begin{prop}
Let $f\fc X\to\R$ and $h\fc X\to \R$. Then
\[ F_f* F_h= F_{f+h}\ .\]
\end{prop}

\begin{proof}
It is sufficient to prove it with $X$ a point. Because of the assumption about the form of $ F_f$, we thus have to compute:
\[\ko_{[a,\infty[}*\ko_{[b,\infty[}=Rs_!(\ko_{[a,\infty[}\boxtimes \ko_{[b,\infty[})=Rs_!(\ko_{[a,\infty[}\times \ko_{[b,\infty[})=\ko_{[a+b,\infty[} \ .\]
\end{proof}

\begin{prop}
Let $f,g\fc X\to \R$ Lipschitz. Then $\partial (f+g) \subset \partial f + \partial g$
\end{prop}

\begin{proof}
Following the previous proposition it is sufficient to bound the microsupport of $\ko_f * \ko_h$. We get:

\[SS((\pi_1^{-1}\ko_f\stackrel{L}{\otimes}\pi_2^{-1}\ko_h))\subset \{ (x,f(x),h(x),-\tau_1\partial f|_x-\tau_2\partial h|_x,\tau_1,\tau_2)\}\] \[\subset T^*(X\times\R\times\R)\]

\noindent with $\tau_i\geq 0$.

We use the bound \ref{cor:opboim} on the microsupport of the direct image for $Rs_!$. The function $s$ is proper on the support of the sheaves $\pi_1^{-1}\ko_f\stackrel{L}{\otimes}\pi_2^{-1}\ko_h$. The Lagrangian $\Lambda_s$ of corollary \ref{cor:opboim} realizes a Lagrangian correspondence between $T^*(X\times\R\times\R)$ and $T^*(X\times\R)$.
Then :

\[\Lambda_s=\{(x,t_1,t_2,x,t,p,\tau_1,\tau_2,p,\tau) |\ t_1+t_2=t, \tau_2=\tau \ \text{and}\ \tau_1=\tau \}\ .\]

We get,

\[SS(Rs_!(\pi_1^{-1}\ko_f\stackrel{L}{\otimes}\pi_2^{-1}\ko_h))\subset \{ (x,f(x)+h(x),-\tau(\partial f|_x+f(x)\partial h|_x),\tau)\} \]\[ \subset T^*(X\times\R)\ .\]
\end{proof}

\begin{rem}
 We have used the bound on the microsupport of tensor product. We should have used the summation $\mathbin{\widehat +}$ of the singular support, but in our case singular supports live in $\{ \tau\geq 0 \}$ and we are interested in the part along $\tau=1$. This remarks allow us to consider only the $+$ part of the summation of singular support.
Indeed, the Lipschitz property of one of the function implies $SS(F_f) \cap SS(F_g)^a\subset \nu^*_{X\times\R}X\times \R$ (see Remark \ref{nicelip}) . 
 \end{rem}

\subsection{Leibniz formula}

\begin{prop}
Let $f\fc X\to\R$ and $h\fc X\to \R$ strictly positive.
\[ F_f\diamond  F_h= F_{f.h} \ .\]
\end{prop}

\begin{proof}
As in the case of the addition, it is sufficient to compute : $\ko_{[a,\infty[}\diamond \ko_{[b,\infty[}$. When $a,b>0$, the hyperbola $ab=C$ for a given $C$ meet $[a,\infty[\times[b,\infty[ $ along a closed arc, which gives the result.
\end{proof}

\begin{rem}
A counterexample can be computed with $f=0$ and $h=1$ :
\[ F_f\diamond  F_h=0\ .\]
\end{rem}

\begin{lemma}\label{opp}
 Let $f\fc X\to \R$ continuous. Then $\partial (-f)=-\partial(f)$
 
\end{lemma}

\begin{proof}
 From \cite{KS} Exercise V.13, we have, because $k_{\epi(f)}$ is cohomologically constructible, that :
 
 \[SS(D(k_{\epi(f)}))=SS(k_{\epi(f)})^a\ .\]
 
Now, according to exercise III.4 of \cite{KS}, we get $D(k_{\epi(f)})=k_{\mathring{\epi(f)}}$.

But we recall we have the following exact triangle :

\begin{displaymath}
    \xymatrix{ k_{\mathring{\epi(f)}} \ar[r] & k_{M\times\R} \ar[r] & k_{f(x)\geq t} \ar[r] & }
\end{displaymath}

We deduce that : \[SS(k_{f(x)\geq t}) \cap \dot{T}(M\times\R)=SS(k_{\epi(f)})^a \cap \dot{T}(M\times\R)\ .\]
After considering the change of variable $t\to (-t)$, we get: \[ SS(\epi(-f))\cap \dot{T}(M\times\R) .\]

\end{proof}

\begin{prop}
Let $f,h\fc X\to \R$  Lipschitz functions. Then :
\[\partial(f.h)|_x \subset h(x).\partial f|_x+f(x)\partial h|_x\ .\]
\end{prop}

\begin{proof}

First we prove the proposition for $f,g$ strictly non negative.

Following the previous proposition it is sufficient to bound the microsupport of $\ko_f\diamond \ko_h$. As for the case of the addition we get :

\[SS((\pi_1^{-1}\ko_f\stackrel{L}{\otimes}\pi_2^{-1}\ko_h))\subset \{ (x,f(x),h(x),-\tau_1\partial f|_x-\tau_2\partial h|_x,\tau_1,\tau_2)\}\] \[\subset T^*(X\times\R\times\R)\]

\noindent with $\tau_i\geq 0$.

We use the bound on the microsupport of the direct image for $Rm_!$. The function $m$ is proper on the support of the sheaves when it is supported in the first quadrant of $\R\times\R$. The Lagrangian $\Lambda_m$ of lemma \ref{cor:opboim} realizes a Lagrangian correspondence between $T^*(X\times\R\times\R)$ and $T^*(X\times\R)$.
Then :

\[\Lambda_m=\{(x,t_1,t_2,x,t,p,\tau_1,\tau_2,p,\tau) |\ t_1t_2=t, \tau_2=\tau t_1\ \text{et}\ \tau_1=\tau t_2 \}\ .\]

We get,

\[SS(Rm_!(\pi_1^{-1}\ko_f\stackrel{L}{\otimes}\pi_2^{-1}\ko_h))\subset \{ (x,f(x)h(x),-\tau(h(x)\partial f|_x+f(x)\partial h|_x),\tau)\} \]\[ \subset T^*(X\times\R)\ .\]

By closure of the microsupport and of the previous bound, the result is also true for $f,g$ positive.

Suppose now without loss of generality that $f$ is negative around some point $x$.
we consider $\tilde f:=-f$. Then \[\partial(fg)|_x=-\partial (\tilde f g)|_x \subset -g(x)\partial \tilde f|_x+f(x) \partial g|_x=g(x)\partial f|_x+f(x) \partial g|_x\]

using $\partial(-f)|_x=-\partial f|_x$

\end{proof}

\begin{exemple}
Let $X=\R$.
\[\partial |x|^2\subset 2 |x| \partial |x|\]
which gives at x=0 :
\[\partial |x|^2|_{x=0}\subset 0\ .\]
\end{exemple}

\subsection{Chain rule}

In this chapter, we would like to understand  the subdifferential of $f\circ L$ with 

\[\left\{
\begin{array}{ll}
 L\fc \R^n\to \R^n \\
 f\fc \R^n \to \R
\end{array}
\right. 
\]

Keeping in mind that graphs ( or epigraphs ) are the good object to microlocally study, we consider the sheaf : $k_{\Gamma_L}\in D^b(\R^n\times\R^n)$.

In the classical study of non-smooth problems the class of graphs is extented to deal with multimap. We will thus consider sets $\Gamma\subset \R^n \times \R^n$.

On the other hand, it is also classical in the field of microlocal analysis of sheaves \cite{KS}  to consider the notion of kernel as elements of $D^b(\R^n\times \R^n)$ 'acting' on $D^b(\R^n)$. Here, we will use a slightly different version of kernel extracted from \cite{Tamarkin}.

\begin{definition}\cite{KS}\cite{Tamarkin}
Let $K_i\in D^b(k_{X\times X\times \R})$, and $\cF\in \cD(k_{X})$, we define:
\[K_2\circ K_1:=R{q_{13s}}_!(q_{124}^{-1} K_1 \stackrel{L}{\otimes} q_{235}^{-1}K_2)\] 
\[\Phi_{K_1}(\cF):=R{q_{1s}}_!(K\stackrel{L}{\otimes} q_{23}^{-1} \cF)\] with canonical maps:

\noindent   $q_{124}\fc X_1\times X_2 \times X_3\times \R_1 \times \R_2\to X_1 \times X_2\times \R_1$

\noindent $q_{235}\fc X_1\times X_2 \times X_3\times \R_1\times \R_2\to X_2 \times X_3\times \R_2$
  
\noindent and $q_{13s}\fc X_1\times X_2 \times X_3\times \R_1\times \R_2\to X_1 \times X_3\times \R $ projection on factor $1,3$ and sum over the $\R$ factor.
 
\noindent $q_{1s}\fc X_1\times X_2\times\R\times\R\to X\times\R$ the sum on $\R\times\R$ and the projection on the first factor.
  
\noindent $q_{23}\fc X_1\times X_2\times\R \to X_2\times\R$ projections over the second and third factor. 

\end{definition}

Following \cite{Tamarkin}, we get the following bound on the microsupport.

\begin{prop}\label{bok} \cite{Tamarkin}
Let $F\in D^b(X\times\R)$ and $K\in D^b(X\times X\times \R)$ satisfying :

\[F*\ko_{[0,\infty[}\simeq F \ and\  K*\ko_{[0,\infty[}\simeq K\ .\]

Then $R(\Phi_K(F))\subset R(F)\circ R(K)$ where $\circ$ denote the composition of Lagrangian correspondence.

\end{prop}

\begin{proof}
We refer to \cite{Tamarkin} Proposition 3.11 for a complete proof. This is mainly based on the bound of \cite{KS} that we recalled in Proposition \ref{cor:opboim}.
\end{proof}

\begin{prop}
Let $f\fc X\to R$ continuous and $L\fc X\to Y$ a $C^1$ map. Then

\[\partial (f\circ L)|_x \subset \partial f|_{f(x)} \circ dL(x)\ .\]
\end{prop}

\begin{proof}
We consider the sheaf $G:=\ko_{\{graph(L) \times [0,\infty[ \}}$. The representative of $G$ is  $graph(dL)$ that we see as the reduction of the conormal of $graph(L)\times \{0\} \times \R$. We thus use Proposition \ref{bok}.
\end{proof}

\begin{rem}
 The generalized Legendre transform of Appendix A fits also in this framework.
\end{rem}

\section{Discussions}

\subsection{Second differential}

Even in the smooth case, there is no canonical second differential for a function define on a manifold. Indeed, this must involve partial derivatives along the $x$ variable for function in $C^\infty(TX,\R)$. But, this would depend on a choice. %
In the case of $X=\R^n$, there exists a canonical integrable connection. 
We thus consider:

\[s_v\fc \R^n\to\mathscr{P}_v=\R^n\times \{v\}\]
\[x\mapsto (x,v)\]

\begin{definition}

Let $f\in C^0(\R^n,\R)$ such that $\partial f$ is compact. We denote by:

\[\mathscr{R}(f)=\{(x,v,t)\in TX\times \R|\ \exists x^*\in (\partial f)|_x\ s.t.\ <x^*,v>\leq t\}\ .\]
 
We name generalized Hessian at $(x,v)\in T\R^n$:

\[\nabla^2(f)(v)|_x:=R(s_v^{-1}k_{\mathscr{R}(f)})^a|_x\subset T_x^*X \ .\]

\end{definition} 

\begin{rem}

The compactness of $\partial f$ is sufficient in order to get $\mathscr{R}(f)$  closed, and such that we can define $k_{\mathscr{R}(f)}$.

\end{rem}

\begin{rem}
We can see $\nabla^2(f)$ as a function from $TX$ to the set of subsets of $T^*X$. Thus to the vector $(x,v)\in T^*\R^n$, $d^2f$, we associate the singleton $\{d^2f_{x}(v,.)\}\subset T^*X$. We call $\nabla^2_{stand}f$ this function.
\end{rem}

\begin{prop}
Let $f\in C^2(\R^{n},\R)$. Then: \[\nabla^2(f)=\nabla^2_{stand}{f} \ .\] 
\end{prop}


\subsection{Graph selector}

The graph selector is a powerful tool defined through symplectic considerations ( generating functions and Floer homology) but with deep impact in Aubry-Mather theory, in the study of Hamilton-Jacobi equations and more generally in symplectic dynamics. Thus, for every smooth exact Lagrangian submanifold $L$, and by a process of selection depending only of Floer homologies $HF^{\leq a}(L,\nu_x^*X)$, it is possible to associate a Lipschitz function $c_L\fc X \to \R$ which is almost everywhere smooth and such that the graph of the differential (if defined) is included into the Lagrangian $L$. The graph selector belongs then to the realm of spectral invariants. In his thesis and in a paper in preparation \cite{whatsheaves}, the author defined a generalization of such invariants for couples of objects in the Tamarkin category $\mathcal D(X)$.

We would like to give a bound on a the subdifferential of a function that has the same properties as the graph selector. We have the following theorem:

\begin{thm}
Let $L$ be a smooth exact Lagrangian of $T^*X$. Then,

\[\partial c_L\subset co(L) \ . \]
\end{thm}

\begin{proof}
We use the fact that the Clarke subdifferential bounds the homological subdifferential.
Moreover, we use Clarke's formula for the subdifferential of Lipschitz function. Indeed, $c_L$ is Lipschitz and its derivatives belongs almost everywhere to $L$. By closedness of $L$, $\partial_c f \subset co(L)$.

\end{proof}

\begin{rem}
 Of course, the Clarke subdifferential satisfies the same property. Nevertheless, it is not adapted to the  study of uniform convergence of graph selector.
\end{rem}

\subsection{Rigidity of Poisson bracket}

The Poisson bracket is defined using differentials of functions. Cardin and Viterbo noticed however that it satisfies some kind of $C^0$ rigidity.

\begin{thm}\cite{CV}
Let $M$ be a symplectic manifold, $f_n,g_n\in C^{1,1}(M,\R)$ such that the Hamiltonian flow of all $g_n$ exists for a given time $t$. If $||\{g_n-g\}||_\infty \to 0$ and $||f_n-f||_\infty\to 0$ and $||\{f_n,g_n\}||_\infty \to 0$ with $f,g\in C^{1,1}(M,\R)$. Then, 

\[\{f,g\}=0 \ .\]
\end{thm}

The following example from \cite{buh} shows that this property is not local and need the existence of the flow for a given time and all $n$.

\begin{exemple}
Take $f_n:=x+\frac{\chi(z)}{\sqrt{n}}\cos(nu)$ and $g_n:=y+\frac{\chi(z)}{\sqrt{n}}\sin(nu)$ defined over an open subset of $(\R^4,x\wedge y+u\wedge z)$ and with $\chi(z)=\sqrt{1-z}$. We then have the following uniform convergence $f_n \to f$ and $g_n \to g$. But, $\{f_n,g_n\}=0$ and $\{f,g\}=1$.
\end{exemple}

However, we can ask the more general question of the behavior of the Poisson bracket in a local chart according to our results on differentials.

\begin{thm}
Let $f_n,g_n\in C^{1}(M,\R)$ and $f,g\in C^{1}(M,\R)$ such that $f_n \to f $ and $g_n \to g$ in the $C^0$ topology. We denote by $Y_{\epsilon,n}=\left\{\{f_n(.+v),g_n(.+v')\}(x),(v,v')\in B^2(0,\epsilon) \right\}$.
Then:
\[\{f,g\}(x)\in \bigcap\limits_{\epsilon>0} \bigcap\limits_{k>0} \overline{\bigcup\limits_{n>k} Y_{\epsilon,n}}\ .\]

\end{thm}

\begin{proof}
The statement is a direct application of the limiting behavior of the subdifferential according to uniform convergence and of the fact that the subdifferential generalizes the derivatives of $C^1$ functions.
\end{proof}

\appendix

\section{Fourrier/Legendre transform} \label{legendre}

A classical tool in convex analysis is the Fenchel-Legendre transform.

\begin{definition}\label{defleg}
Let $f\fc \R^n\to \R$. Its Fenchel-Legendre transform $f^*$ is given by :

\[ f^*(k)=\inf\limits_{x\in\R^n}(-\langle k,x \rangle+f(x))\ .\]
\end{definition}

The definition of such a transformation is not well posed according to sheaf operation. But, sheaf theoretic methods are also interesting to study envelopes and in the case of proper convex function :

\begin{prop}
Let $K=\ko_{-\{\langle x,y \rangle \leq t\}}$ and $f$  superlinear convex function . Then :

\[K\circ F_f\simeq \ko_{\epi(f^*)}\ .\]

\end{prop}

\begin{proof}
 Let us compute the germ at $(y_0,t_0)$ of the right hand side given by :

 \[ H^*_c(f(x)-\langle y_0,x\rangle \leq t_0) \ .\]
 
The set $\{x, f(x)-\langle y_0,x\rangle \leq t_0\}$ is compact and convex. The germ is then either $\ko$ if $(y_0,t_0)\in \epi(f^*)$ either $0$. 
\end{proof}

\begin{definition}
We define the extended Legendre transform of a function $f$ to be the sheaf defined by:

\[\hat{f}=K\circ F_f \ .\]
We define the opposite extended Legendre transform of a sheaf $F$ by:

\[\overline{F}=L\circ F\]
\noindent with $L=\ko_{\{\langle x,y\rangle\leq t\}}$.
\end{definition}

\begin{exemple}
Let $f\fc \R^n \to \R$, $x\mapsto \langle v,x \rangle$ for the standart scalar product. Then, 

\[\hat{f} \simeq \ko_{v}[n]\ .\]

\end{exemple}

We have the following inversion

\begin{prop} \label{invert}
 Let $f\fc \R^n\to \R$ be a Lipschitz function. Then,
 
 \[ \overline{\hat{f}}=\ko_{\epi(f)}[n] \ .\]
 
\end{prop}

\begin{proof}
 We refer here to the article of \cite{Tamarkin} thm 3.5. It is mainly based on associativity of kernel composition and to the explicit computation of $K\circ L$.
\end{proof}

We thus deduce :

\begin{prop} \label{rotation} 
Let $f$ be a Lipschitz function. Then the representative of $\hat{f}$ is the closure of the canonical rotation by $-\frac{\pi}{2}$ of $\partial f$ in $T^*\R^n$.
\end{prop}

\begin{proof}
 Let's denote by $J$ the previous canonical rotation in $T^*\R^n $. Because $f$ is Lipschitz, $\partial f$ is closed and we can apply thm 3.6 in \cite{Tamarkin}. Then,
 
 \[R(\hat(f)\subset J(\partial f) \ .\]
 
 In the other and applying the same theorem to the kernel $L$ gives and according to \ref{invert}:
 
 \[\partial(f)\subset J^{-1} \overline{R(f)} \ . \]
 
 Coupling the two inclusions:
 
 \[J\partial f\subset J.J^{-1} \overline{R(\hat{f})}=\overline{\hat{f}}\subset J \partial f\ .\]

 \noindent We then deduce :
 
 \[\overline{R(\hat {f})}=J \partial f \ .\]

\end{proof}

\begin{exemple}
 Here, we deal with the example of a non-convex function $f(x):=-x^4+2x$.
 The graph of $f$ is the following.
 
 \begin{center}
 \includegraphics[width=0.30\textwidth]{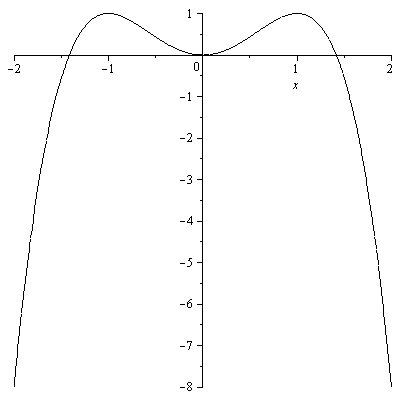}
 \end{center}
 The graph of the differential and its rotation by angle $-\frac{\pi}{2}$.
 
 \begin{center}
 \includegraphics[width=0.30\textwidth]{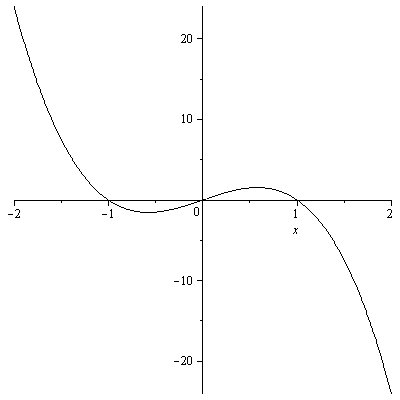}
 \end{center}
 
 \begin{center}
 \includegraphics[width=0.30\textwidth]{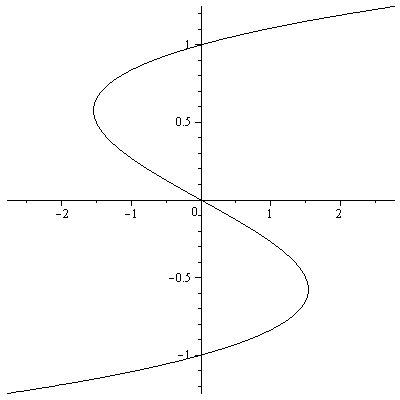}
 \end{center}
  
  Thus, the extended Legendre transform of $f$ is the sheaf represented by the next picture. We quote the germ of the sheaf, it can be considered as the right extension between two constant sheaves on locally closed set.
  
  \begin{center}
  \includegraphics[width=0.30\textwidth]{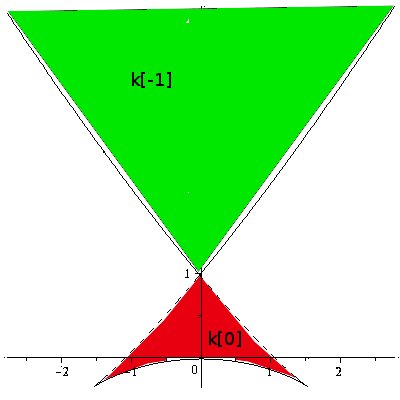}
  \end{center}
  
  It is a general fact that the extended Legendre transform is related to the front of the rotation of $graph(df)$ and can be considered as a decorated front. Moreover, the traditional Legendre transform (of the convexified or concavified) can be read of the front as a graph selector.

\end{exemple}


\begin{thebibliography}{2}
  \bibitem[B]{buh} Buhovski, L; The $\frac{2}{3}$-convergence rate for the Poisson bracket. Geometric and Functional Analysis, Volume 19, Number 6 (2010), 1620-1649.
  \bibitem[CV]{CV} Cardin, F; Viterbo, C; Commuting Hamiltonians and Hamilton-Jacobi multi-time equations, Duke Math. J. Volume 144, Number 2 (2008), 235-284
  \bibitem[Cl]{Clarke} Clarke, F.H.; Optimization and Nonsmooth Analysis, SIAM
  \bibitem[Fa]{Fathi} Fathi, A; Weak KAM Theorem in
Lagrangian Dynamics, {\it Preliminary version}
  \bibitem[GS]{GS} Guillermou, S; Schapira, P; Microlocal theory of sheaves and Tamarkin's non displaceability theorem
  \bibitem[Hum]{Humiliere} Humiliere, V; On some completions of the space of Hamiltonian maps, Bulletin de la soci\'et\'e math\'ematique de France 136, 3 (2008) 373-404
  \bibitem[I]{Ioffe} Ioffe, A,D; On the theory of subdifferentials ,Advances in Nonlinear Analysis. Volume 1, Issue 1, Pages 47-120.
  \bibitem[I2]{Ioffe2} Ioffe, A,D;  Subdifferentials and Applications. I: The Finite Dimensional Theory
, Transactions of the American Mathematical Society, Vol. 281, No. 1 (Jan., 1984), pp.
389-416

   \bibitem[J]{jourani} Jourani A.; Limit Superior of Subdifferentials of Uniformly Convergent Functions, Positivity March 1999, Volume 3, Issue 1, pp 33-47
   \bibitem[KS]{KS} Kashiwara M; Schapira, P; Sheaves on manifolds, Springer, 2002.
   \bibitem[MVZ]{MVZ} Monzner, A; Vichery, N; Zapolsky, F; Partial quasimorphisms and quasistates on cotangent bundles, and symplectic homogenization, Journal of modern dynamics Vol 6, n2.
   \bibitem[Tam]{Tamarkin} Tamarkin, D; Microlocal condition for non-displaceablility , {\it Ar$\chi$iv:0809.1584}.
   \bibitem[Vic]{whatsheaves} Vichery, N; Symplectic topology as sheaf theory. {\it in preparation}
   \bibitem[Vic2]{nonconvex} Vichery, N; Spectral invariants towards a non convex Mather theory. {\it in preparation}
   \bibitem[Vit]{Viterbo_course_eil} Viterbo, C; An Introduction to Symplectic Topology through Sheaf theory.
\end{thebibliography}
\end{document}